\documentclass[letterpaper,10pt]{amsart}
\usepackage[utf8]{inputenc}
\usepackage{graphicx} 

\usepackage{amsmath, amssymb, amsthm}
\usepackage{mathtools}
\usepackage{tikz-cd}
\usepackage{soul}
\usepackage[hidelinks]{hyperref}

\def\cal{\mathcal}

\newcommand{\cQ}{\mathcal{Q}}
\newcommand{\Z}{\mathcal{Z}}
\newcommand{\W}{\mathcal{W}}

\DeclareMathOperator{\id}{id}

\newcommand{\norm}[1]{{\left\lVert #1\right\rVert}}

\theoremstyle{plain}
\newtheorem{main}{Theorem}
\newtheorem{mcor}[main]{Corollary}
\newtheorem{thm}{Theorem}[section]
\newtheorem*{thm*}{Theorem}
\newtheorem{lem}[thm]{Lemma}
\newtheorem{prop}[thm]{Proposition}

\newtheorem{cor}[thm]{Corollary}

\theoremstyle{definition}
\newtheorem{defn}[thm]{Definition}
\newtheorem*{defn*}{Definition}

\newtheorem{remark}[thm]{Remark}

\newtheorem*{question*}{Question}

\usepackage[
backend=biber,
style=alphabetic,
sorting=nyt,
maxbibnames=9
]{biblatex}
\title{Embeddings into the ultrapower of the Jiang-Su algebra}
\author{Ben Bouwen}
    \address{Department of Mathematics and Computer Science, University of Southern Denmark, Odense, Denmark}
    \email{bouwen@imada.sdu.dk}
\author{Jennifer Pi}
    \address{Department of Mathematics, University of Oxford, Oxford, United Kingdom}
    \email{jennifer.pi@maths.ox.ac.uk}

\thanks{This work was supported by DFF grant 1054-00094B, and by the the ERC grant EP/X026647/1.}

\begin{document}

\begin{abstract}
    We study existence of embeddings into ultrapowers of the Jiang-Su algebra $\Z$ and the Razak-Jacelon algebra $\W$. More specifically, we show that the cone over any separable $C^*$-algebra embeds into the ultrapowers of both $\Z$ and $\W$. We also show that the result for $\Z$ generalizes to separable and exact continuous fields of $C^*$-algebras for which one of the fibers embeds into the ultrapower of $\Z$, if this fiber is suitably well-behaved.
\end{abstract}

\maketitle

\section{Introduction}

Questions of embeddability into ultrapowers within operator algebras are ubiquitous, with the most famous being the Connes embedding problem: does every II$_1$ factor embed into an ultrapower of the hyperfinite II$_1$ factor?
A naive $C^*$-analog of this question, regarding embeddings of $C^*$-algebras into an ultrapower of the Jiang-Su algebra $\Z$, can be quickly decided: any $C^*$-algebra with multiple non-zero projections cannot embed into $\Z_\omega$, since the latter is projectionless. 
However, the question of characterizing which $C^*$-algebras do admit embeddings into $\Z_\omega$ remains open and difficult. 

In this brief paper, we make a few observations on embeddability of various $C^*$-algebras into the ultrapower of $\Z$ and of the Razak-Jacelon algebra $\W$, a stably projectionless $C^*$-algebra with trivial $K$-theory constructed in \cite{Jacelon}. The first asserts existence of an embedding for cones over separable $C^*$-algebras:

\begin{main}
    The cone over any separable $C^*$-algebra embeds into $\Z_\omega$ and into $\W_\omega$.
\end{main}

This statement is obtained as a consequence of Proposition \ref{cone over Q_omega}. In the latter, we construct embeddings of the cone over the universal UHF algebra $\cQ$ into $\Z$ and $\W$, which are used to embed the cone over $\cQ_\omega$ into $\Z_\omega$ and $\W_\omega$, respectively. The statement above then follows by combining quasidiagonality of arbitrary cones over separable $C^*$-algebras \cite{Voi1991QD} with an embedding of the cone into the double cone.

Combining Proposition \ref{cone over Q_omega} with the results in \cite{gabe2020traceless}, we add $\Z_\omega$- and $\W_\omega$-embeddability to the following list of equivalences for separable, exact, traceless (in the sense of \cite{gabe2020traceless}) $C^*$-algebras.

\begin{mcor}
    For $A$ separable, exact, and traceless, the following are equivalent:
    \begin{enumerate}
        \item $A$ embeds into the cone over $\cal{O}_2$,
        \item $A$ is quasidiagonal.
        \item $A$ is AF-embeddable.
        \item The primitive ideal space of $A$ has no non-empty compact open subsets.
        \item $A$ is stably finite.
        \item $A$ is stably projectionless.
        \item $A$ embeds into the trace-kernel ideal of $\Z_\omega$.
        \item $A$ embeds into the trace-kernel ideal of $\W_\omega$.
    \end{enumerate}    
\end{mcor}

Our second result generalizes Theorem A to continuous fields of $C^*$-algebras under some additional conditions:

\begin{main}
    Let $E$ be a continuous field of $C^*$-algebras over a connected, compact, metrizable space $X$. Suppose $E$ is separable and exact, and one of the fibers is simple, nuclear and $\Z_\omega$-embeddable. If this fiber is non-unital, or if $E$ is unital, then $E$ is $\Z_\omega$-embeddable. In the latter case, if the embedding of the fiber is unital, then so is the embedding of $E$.
\end{main}

This follows from Theorem \ref{thm:extensionsembed}, which asserts that $\Z_\omega$-embeddability passes to extensions. The proof makes use of extension theory, and specifically Elliott and Kucerovsky's characterization of unitally absorbing extensions \cite{Elliott-Kucerovsky_2001}. As a particular consequence, we show that homotopy equivalence to $\Z$ is a sufficient condition to embed into $\Z_\omega$ for separable, exact $C^*$-algebras:

\begin{mcor}
    Any separable, exact $C^*$-algebra $A$ which is homotopy equivalent to $\Z$, embeds into $\Z_\omega$.
\end{mcor}

\subsection*{Acknowledgements} 
This work started during visits of the first author to the University of Oxford in 2024. We would like to thank Stuart White for his hospitality during this time and for his guidance throughout the project, and James Gabe for several very insightful conversations and remarks. We also want to thank Christopher Schafhauser for the illuminating ideas and discussions.

\section{Existence of embeddings into ultrapowers of $\Z$ and $\W$}
For an arbitrary $C^*$-algebra $A$, we denote the cone over $A$ as
\begin{equation}
    CA := C_0(0,1] \otimes A \cong C_0((0,1],A).
\end{equation}
We also recall the definition of the ultrapower and the trace-kernel ideal of a $C^*$-algebra. For the entirety of this paper, fix a free ultrafilter $\omega$ on $\mathbb N$.
\begin{defn}
    Let $B$ be a $C^*$-algebra. The \emph{ultrapower} of $B$ is defined as
    \begin{equation}
        B_\omega := \ell^\infty(B) /\{ (x_n)_n \in \ell^\infty(B) \mid \lim_{n\to \omega} \|x_n\| = 0\}.
    \end{equation}
    If moreover the set of tracial states $T(B)$ on $B$ is non-empty, the \emph{trace-kernel ideal} of $B$ is defined as
    \begin{equation}
        J_B := \{ (x_n)_n \in B_\omega \mid \lim_{n\to \omega}\textstyle{\sup_{\tau\in T(B)}} \tau(x_n^*x_n) = 0\}.
\end{equation}
\end{defn}

\begin{prop} \label{cone over Q_omega}
    The cone over any UHF algebra embeds into $\Z$ and into $\W$. Subsequently, the cone over $\cQ_\omega$ embeds into $\Z_\omega$ and $\W_\omega$.
\end{prop}
\begin{proof}
Let $M_n$ be a UHF algebra corresponding to a supernatural number $\mathfrak n$. We first show that $CM_{\mathfrak n}$ embeds into $\Z$. Applying \cite[Proposition 3.3]{RordamWinterJiangSu} for $p = 1$ and $q = \mathfrak n$ gives a unital embedding of
\begin{equation}
    Z_{1,\mathfrak n} := \{ f\in C([0,1], \mathbb C \otimes M_{\mathfrak n} ) \mid f(0) \in \mathbb C \otimes \mathbb C,\ f(1)\in \mathbb C\otimes M_{\mathfrak n} \} \cong (C M_{\mathfrak n})^\sim, 
\end{equation}
the unitization of $C M_{\mathfrak n}$ into $\Z$. In particular, $CM_{\mathfrak n}$ embeds into $\Z$.

To show that $CM_{\mathfrak n}$ also embeds into $\W$, write $M_{\mathfrak n}$ as an inductive limit $\varinjlim M_{n_k}$ of matrix algebras with connecting maps $\mu_k: M_{n_k} \to M_{n_{k+1}}$ and $n_k | n_{k+1}$ for all $k$. As such, $CM_{\mathfrak n}$ is the inductive limit $\varinjlim C M_{n_k}$. Consider the Jacelon building blocks as in \cite{Jacelon}: for $\ell, m \in \mathbb{N}$,
\begin{equation}
    W(\ell,m) := \{ f\in C([0,1], M_{\ell+1}\otimes M_m) \mid f(0) = (1_\ell \oplus0)\otimes c, \ f(1) = 1_{\ell+1} \otimes c \text{ for some } c\in M_m \} 
\end{equation}
and note that $W(\ell,m) \cong W(\ell,1) \otimes M_m$. Hence, the embedding
\begin{equation}
    \iota: C_0(0,1] \to W(1,1): f \mapsto f(1) \oplus f
\end{equation}
induces embeddings $\iota\otimes \id_{M_{n_k}}: C M_{n_k} \to W(1, n_k)$  for all $k\in\mathbb N$. Now let
\begin{equation}
    \rho_k := \id_{W(1,1)} \otimes \mu_k: W(1, n_k) \to W(1,n_{k+1}),
\end{equation}
and note that $\rho_k$ restricts to $\id_{C_0(0,1]}\otimes \mu_k$ on $CM_{n_k}$ and preserves the faithful trace $\tau_k$ on $W(1, n_k)$ induced by the faithful trace on $C[0,1] \otimes M_2\otimes M_{n_k}$ coming from Lebesgue integration on $[0,1]$. Use part (i) of \cite[Lemma 4.1]{Jacelon} to find $*$-homomorphisms $\psi_k: W(1, n_k) \to \W$ which satisfy $\tau \circ \psi_k = \tau_k$, where $\tau$ is the unique trace on $\cal{W}$. In particular, this forces the $\psi_k$ to be injective. Moreover, part (ii) of \cite[Lemma 4.1]{Jacelon} implies that $\psi_k$ and $\psi_{k+1}\circ \rho_k$ are approximately unitarily equivalent. A one-sided intertwining argument (see for example \cite[Theorem 1.10.14]{LinClassification}) then shows that there exists a copy of the inductive limit $\varinjlim (W(1, n_k),\rho_k)$ inside of $\W$, and since the $\rho_k$ respect the connecting maps $\id_{C_0(0,1]}\otimes \mu_k: CM_{n_k} \subseteq W(1, n_k)$, $\W$ also contains $\varinjlim C M_{n_k} \cong C M_{\mathfrak n}$.

For the second part of the statement, note that the embedding $C\cQ \hookrightarrow \Z$ from the previous part of the proof induces an embedding
\begin{equation}
    C \cQ_\omega \subseteq \big( C \cQ\big)_\omega \hookrightarrow \Z_\omega.
\end{equation}
This concludes the proof.
\end{proof}

\begin{remark}
    As a consequence of the main theorem in \cite{Robert_NCCWclassification}, there exists a trace-preserving embedding of $\W$ into $\Z$, which results in an embedding $J_\W \hookrightarrow J_\Z$. This means that the embedding into $J_\Z$ can alternatively be obtained through its embedding into $J_\W$.
\end{remark}

\begin{defn}
    A $C^*$-algebra $A$ is said to be \emph{traceless} if every lower semicontinuous $2$-quasitrace can only take the values $0$ or $\infty$.
\end{defn}

\begin{cor} \label{cone over separable}
    The cone over any separable $C^*$-algebra $A$ embeds into $\Z_\omega$ and into $\W_\omega$. If $CA$ is traceless, then the embeddings lie within the trace-kernel ideals $J_\Z$ resp.\ $J_\W$.
\end{cor}
\begin{proof}
Let $A$ be a separable $C^*$-algebra. As $CA$ is quasidiagonal by \cite{Voi1991QD}, there exists an embedding $CA \hookrightarrow \cQ_\omega$, inducing an embedding $C(CA) \hookrightarrow C \cQ_\omega$. Now, the map $\iota: CA \to C(CA)$ given by
\begin{equation}
    \iota(f): [0,1]^2\to A: (t,s) \mapsto f(ts) \quad (f\in CA)
\end{equation}
is an injective $*$-homomorphism, so composing it with the map $C(CA) \hookrightarrow C \cQ_\omega$ gives an embedding $CA\hookrightarrow C\cQ_\omega$. The result then follows from Proposition \ref{cone over Q_omega}.

For the second part of the statement, assume $CA$ is traceless. Note that any trace on $\Z_\omega$ must be identically zero on $CA$, as such a trace would induce a non-zero, finite trace on $CA$ otherwise. In particular, this holds for the limit trace induced by the unique trace on $\Z$, which implies that the image of $CA$ must land in $J_\Z$. The same argument holds for $\W_\omega$, finishing the proof.
\end{proof}

The corollary above is an analogue of \cite[Theorem 8.3.5]{BrownOzawa_book} for $\Z_\omega$- and $\W_\omega$-embeddability, showing that both properties are similar in flavor to AF-embeddability. With this in mind, we want to highlight what happens in the traceless case. In this setting, all these notions of embeddability turn out to be equivalent, as a consequence of the results in \cite{gabe2020traceless} and Corollary \ref{cone over separable}:

\begin{cor} \label{equivalences}
    For $A$ separable, exact, and traceless, the following are equivalent:
    \begin{enumerate}
        \item \label{embeds in cone}$A$ embeds into $C_0(0,1] \otimes \cal{O}_2$.
        \item \label{qd} $A$ is quasidiagonal.
        \item $A$ is AF-embeddable.
        \item $\operatorname{Prim}(A)$ has no nonempty compact open subsets.
        \item \label{stably finite} $A$ is stably finite.
        \item $A$ is stably projectionless.
        \item \label{embeds in J_Z}$A$ embeds into $J_\Z \subseteq \Z_\omega$.
        \item \label{embeds in J_W}$A$ embeds into $J_\W \subseteq \W_\omega$.
    \end{enumerate}
\end{cor}
\begin{proof}
    The equivalence between the first four items above is \cite[Theorem A]{gabe2020traceless}, and Corollary C from the same paper adds (5) and (6).
    The implications (\ref{embeds in cone}) $\implies$ (\ref{embeds in J_Z}) and (\ref{embeds in cone}) $\implies$ (\ref{embeds in J_W}) are Corollary \ref{cone over separable} for $A = \cal O_2$.
    
    (\ref{embeds in J_Z}) $\implies$ (\ref{stably finite}): As $\Z$ is stably finite, its ultrapower $\Z_\omega$ is stably finite by \cite[\S 3.6(d)]{farah2021model}. It follows that $A$ is also stably finite, being a $C^*$-subalgebra of the stably finite $C^*$-algebra $\Z_\omega$.

    (\ref{embeds in J_W}) $\implies$ (\ref{stably finite}): As $\W$ is stably projectionless and therefore stably finite, the same argument as above does the trick.
\end{proof}

\section{Extensions and continuous fields of $C^*$-algebras}
In this section, we prove our second result, that certain extensions of $\cal{Z}_\omega$-embeddable $C^*$-algebras remain $\cal{Z}_\omega$-embeddable. The idea was inspired by its analogue for AF-embeddability (see for example \cite[Proposition 8.4.9]{BrownOzawa_book}), but the approach taken needs to be quite different, as $\Z_\omega$-embeddability is not preserved under stabilizations in general. Our proof makes extensive use of extension theory and specifically Elliott and Kucerovsky's characterization of unitally absorbing extensions \cite{Elliott-Kucerovsky_2001}. We give a brief overview of the relevant concepts; for a more in-depth discussion as well as proofs of the assertions made, see for example \cite[Chapter VII]{blackadar1998k} or \cite[\S 2]{gabeRuiz2020}.

\subsection*{Extension theory}
An \emph{extension} of a $C^*$-algebra $A$ by a $C^*$-algebra $I$ is a short exact sequence of the form
\begin{equation}
\begin{tikzcd}
        \mathfrak e: 0 \arrow[r] & I \arrow[r] & E \arrow[r] & A \arrow[r] & 0.
\end{tikzcd}
\end{equation}
An extension is said to be \emph{split} if it has a $*$-homomorphic splitting, and \emph{unital} if $E$ (and therefore necessarily also $A$) is unital. Extensions of $A$ by $I$ correspond one-to-one (up to congruence of extensions) to $*$-homomorphisms $A\to \cQ(I)$ (where $\cQ(I) := \mathcal{M}(I)/I$ is the corona algebra of $I$), which are known as their \emph{Busby maps}. Two extensions $\mathfrak e_1$ and $\mathfrak e_2$ of $A$ by $I$ are said to be \emph{strongly unitarily equivalent}, denoted $\mathfrak e_1\sim_s\mathfrak e_2$, if their corresponding Busby maps are unitarily equivalent by a unitary descending from a unitary in $\mathcal{M}(I)$; in this case, the extension algebras are also isomorphic.

We will only be concerned with the situation where $I$ is $\sigma$-unital and $A$ is separable. If $I$ is stable, then one can define a direct sum $\oplus$ within $\cal M(I)$ (uniquely up to unitary equivalence), which allows us to define an addition $\oplus$ on the set of extensions of $A$ by $I$. As such, one obtains an abelian monoid $\mathrm{Ext}(A,I)$ of extensions with this direct sum as binary operation, modulo the equivalence relation where two extensions $\mathfrak e_1$ and $\mathfrak e_2$ of $A$ by $I$ are equivalent if $\mathfrak e_1\oplus \mathfrak f \sim_s \mathfrak e_2 \oplus \mathfrak f$ for some split extension $\mathfrak f$ of $A$ by $I$. The neutral element of $\mathrm{Ext}(A,I)$ then becomes the equivalence class of split extensions. If moreover $A$ is nuclear, then $\mathrm{Ext}(A,I)$ is in fact a group.

In this context, an extension $\mathfrak e$ is said to be \emph{absorbing} if $\mathfrak e\oplus \mathfrak f\sim_s\mathfrak e$ for all split extensions $\mathfrak f$. Similarly, a unital extension $\mathfrak e$ is said to be \emph{unitally absorbing} if the same holds for all $\mathfrak f$ with a unital splitting. Remarkably, Elliott and Kucerovsky characterized unitally absorbing extensions by the following inherent property, in \cite{Elliott-Kucerovsky_2001}. Recall that, for two positive elements $a$ and $b$ of a $C^*$-algebra $A$, $a$ is said to be \emph{Cuntz subequivalent} to $b$, denoted $a \precsim b$, if for every $\varepsilon>0$, there exists a $v\in A$ such that $v^*bv \approx_\varepsilon a$. Here, we denoted $x \approx_\varepsilon y$ for elements $x$ and $y$ in a $C^*$-algebra $A$ to mean $\norm{x-y} < \varepsilon$.

\begin{defn}
    An extension
    \begin{equation}
    \begin{tikzcd}
        \mathfrak e: 0 \arrow[r] & I \arrow[r] & E \arrow[r] & A \arrow[r] & 0
    \end{tikzcd}
    \end{equation}
    of $C^*$-algebras is said to be \emph{purely large} if for each $x\in I_+$ and $y\in E_+\setminus I$, we have $x\precsim y$.
\end{defn}

This is not the original definition of pure largeness given by Elliott and Kucerovsky. However, it is shown to be equivalent to pure largeness when the ideal is $\sigma$-unital and stable in \cite[Proposition 4.14]{BouwenGabe} (the proof of which is attributed to Carrión, Gabe, Schafhauser, Tikuisis and White).

\begin{lem} \label{lemma:unitizationpurelylarge}
    Let $J$ be a $\sigma$-unital, stable $C^*$-algebra and $D$ a $\sigma$-unital, simple and exact $C^*$-algebra. Then the extension
    \begin{equation}
    \begin{tikzcd}
        \mathfrak f: 0 \arrow[r] & J \otimes D \arrow[r] & J^\dagger \otimes D \arrow[r] & D \arrow[r] & 0
    \end{tikzcd}
    \end{equation}
    obtained by applying $\cdot \otimes D$ to the unitization short exact sequence of $J$, is purely large.
\end{lem}
\begin{proof}
We start by showing that $\mathfrak f$ is purely large. As $J$ and $D$ are both $\sigma$-unital, they contain strictly positive elements $a\in J_+$ resp.\ $b\in D_+$, with $\norm{a} = \norm{b} = 1$. As $x \precsim a\otimes b$ for all $x\in (J\otimes D)_+$ by \cite[Proposition 2.7.ii]{kirchbergRordam2000}, it suffices to show that $a\otimes b \precsim y$ for all $y\in (J^\dagger\otimes D)_+\setminus J\otimes D$. Write $y=z+1\otimes d$ for some $z\in J\otimes D$ and $0\neq d\in D_+$. Fix an isomorphism $J\otimes \cal K \cong J$, let $\{f_{i,j}\}_{i,j \in \mathbb N}$ be matrix units in $\cal K$, $\{g_n\}_n$ a sequential approximate unit for $J$ such that $g_{n+1} g_n = g_n$ for all $n$ and denote $e_n := \sum_{i=1}^n g_n\otimes f_{i,i} \in J\otimes \cal K \cong J$ ($n\in\mathbb N$).

First, we claim that for each $\varepsilon>0$ and $n\in \mathbb N$, there exist $k\geq n$ and a contraction $v\in J^\dagger\otimes D$ such that $v^*yv \approx_\varepsilon (1-e_k)\otimes d$. To see this, pick $\varepsilon$ and $n$ as such. Take $k\geq n$ and $\tilde d \in D_+$ an appropriate element of an approximate unit for $D$ such that $d\approx_\varepsilon \tilde d d \tilde d$ and $z \approx_\varepsilon (e_{k-1}^{1/2}\otimes \tilde d)z(e_{k-1}^{1/2}\otimes \tilde d)$. Then $v := (1-e_k)^{1/2}\otimes \tilde d$ is a positive contraction such that
\begin{equation}
    v^*yv \approx_{\varepsilon} v(e_{k-1}^{1/2}\otimes \tilde d)z(e_{k-1}^{1/2}\otimes \tilde d)v + (1-e_k)\otimes \tilde d d \tilde d \approx_{\varepsilon} (1-e_k)\otimes d,
\end{equation}
proving our claim.

Now, to show that $a\otimes b\precsim y$, fix $\varepsilon>0$, and take $n\in\mathbb N$ such that $a\approx_{\varepsilon} a^{1/2}e_n a^{1/2}$. As $D$ is simple, there exist $b_1, ..., b_m \in D$ such that $b \approx_{\varepsilon/\|a^{1/2}e_n a^{1/2}\|} \sum_{j=1}^m b_i^*db_i$. Applying our claim for $\frac{\varepsilon}{\|a^{1/2}e_n a^{1/2}\|\sum_{j=1}^m \|b_j\|^2}$ and $n$, we obtain $k\geq n$ and $v \in J^\dagger\otimes D$. Remark that the elements
\begin{equation}
    h_i := \sum_{j=1}^k (e_{(2i+1)k}-e_{2ik})^{1/2}(g_j^{1/2} \otimes f_{2ik+j,j})e_n^{1/2} \in J\otimes \cal K \cong J \quad (i\in \{1,...,m\}
\end{equation}
are contractions such that $h_i^*h_j = h_i^*(1-e_k)h_j = \delta_{ij} e_n$ for all $i,j\in\{1,...,m\}$. The element
\begin{equation}
    w :=\sum_{j=1}^m h_j a^{1/2} \otimes b_j \in J\otimes D
\end{equation}
then satisfies
\begin{equation}
    w^*w = \sum_{i,j} (a^{1/2}h_i^*h_ja^{1/2})\otimes b_i^*b_j = a^{1/2} e_n a^{1/2}\otimes \sum_{j=1}^m b_j^*b_j,
\end{equation}
so
\begin{equation}
    \|w\|^2 = \norm{a^{1/2} e_n a^{1/2}\otimes \sum_{j=1}^m b_j^*b_j} \leq \norm{a^{1/2} e_n a^{1/2}}\cdot \sum_{j=1}^m \|b_j\|^2.
\end{equation}
Similarly, we have
\begin{equation}
    w^*\big((1-e_k)\otimes d \big)w = a^{1/2} e_n a^{1/2} \otimes \sum_{j=1}^m b_j^* d b_j \approx_{2\varepsilon} a\otimes b,
\end{equation}
so we can conclude that
\begin{equation}
    (vw)^*yvw \approx_\varepsilon w^*\big((1-e_k)\otimes d \big)w \approx_{2\varepsilon} a\otimes b,
\end{equation}
as desired.
\end{proof}

\begin{prop} \label{prop: embedding tensor product}
    Let $A$, $B$, $C$ and $D$ be $C^*$-algebras with $C$ and $D$ unital. If the canonical surjection $A\otimes_{\max} B \to A\otimes B$ is an isomorphism (in particular, if $A$ or $B$ is nuclear), and there exist embeddings $A \hookrightarrow C_\omega$ and $B\hookrightarrow D_\omega$, then $A\otimes B$ embeds into $(C\otimes_\alpha D)_\omega$ for any $C^*$-norm $\|\cdot\|_\alpha$ on $C\odot D$. Moreover, if $A$ and $B$ are unital and the given embeddings are unital, then the resulting embedding of $A\otimes B$ will also be unital.
\end{prop}
\begin{proof}
Let $\iota_A: A\hookrightarrow C_\omega$ and $\iota_B: B\hookrightarrow D_\omega$ be the given embeddings, and consider
\begin{equation}
\begin{aligned}
    \varphi_C&: C_\omega\to (C\otimes_\alpha D)_\omega: (c_n)_n \mapsto (c_n\otimes 1_D)_n, \\
    \varphi_D&: D_\omega\to (C\otimes_\alpha D)_\omega: (d_n)_n \mapsto (1_C\otimes d_n)_n.   
\end{aligned}
\end{equation}
These are well-defined $*$-homomorphisms with commuting ranges. Indeed, the fact that they are well-defined follows from $\|\cdot\|_\alpha$ being a cross norm \cite[Lemma 3.4.10]{BrownOzawa_book}; the other properties are immediate. By universality of the maximal tensor product and the isomorphism $A\otimes B \cong A\otimes_{\max} B$, the maps $\varphi_C \circ\iota_A$ and $\varphi_D \circ \iota_B$ then induce a $*$-homomorphism
\begin{equation}
    \psi := (\varphi_C \circ\iota_A) \times (\varphi_D \circ \iota_B): A\otimes B \to (C\otimes_\alpha D)_\omega: a\otimes b \mapsto (\iota_A(a)_n\otimes \iota_B(b)_n)_n,
\end{equation}
where $\iota_A(a)_n$ denotes the $n$-th element of a representatitive sequence of $\iota_A(a)$ (and similarly for $\iota_B(b)_n$). We now claim that this map is injective. To see this, note that it suffices to show injectivity on elementary tensors, since any non-zero ideal $I$ in $A \otimes B$ contains a non-zero elementary tensor. Indeed, by Kirchberg's slice lemma \cite[Lemma 4.1.9]{rordam2001classification}, there exists a non-zero $z\in A\otimes B$ such that $z^*z$ is an elementary tensor and $zz^*\in I$. By polar decomposition in $C^*$-algebras, there exists $w\in A\otimes B$ such that $z^* = w(zz^*)^{1/4} \in I$, so $I$ contains the non-zero elementary tensor $z^* z$.

Take $a\in A$ and $b\in B$ such that $\psi(a\otimes b) = 0$, and denote $(a_n)_n := \iota_A(a)$ and $(b_n)_n := \iota_B(b)$. Recalling that $\|\cdot\|_\alpha$ is a cross norm, we then find
\begin{equation}
    0 = \lim_{n\to \omega} \|a_n\otimes b_n\|_\alpha = \lim_{n\to \omega} \|a_n\|\cdot \| b_n\| = \lim_{n\to \omega} \|a_n\|\cdot \lim_{n\to \omega} \| b_n\|.
\end{equation}
Hence, $(a_n)_n$ or $(b_n)_n$ has to tend to zero along $\omega$. This means that $\iota_A(a) = 0$ or $\iota_B(b) = 0$, which implies that $a=0$ or $b=0$. We conclude that $\psi$ is indeed injective.

In the setting where $A$, $B$, $\iota_A$ and $\iota_B$ are all unital, it is clear from the definitions that the same holds for $\varphi_C$ and $\varphi_D$, and therefore also for $\psi$. This finishes the proof.
\end{proof}

We are now ready to prove the main result of this section.

\begin{thm} \label{thm:extensionsembed}
    Let
    \begin{equation}
    \begin{tikzcd}
        \mathfrak e: 0 \arrow[r] & I \arrow[r] & E \arrow[r] & D \arrow[r] & 0
    \end{tikzcd}
    \end{equation}
    be an extension of $C^*$-algebras with $I$ such that $\mathrm{Prim}(I)$ contains no non-empty compact open subsets, $E$ separable and exact, and $D$ simple and nuclear. If $D$ is non-unital and embeds into $\cal Z_\omega$, or if $E$ (and therefore also $D$) is unital and $D$ embeds into $\cal Z_\omega$, then so does $E$. In the latter case, $E$ embeds unitally if $D$ does.
\end{thm}
\begin{proof}
As $I$ is separable and exact, it follows from \cite[Theorem A]{gabe2020traceless} that $I$ embeds into R\o rdam's ASH-algebra $\cal A_{[0,1]}$ (see \cite{Rordam_ASHalgebra}). Letting $J$ be the hereditary $C^*$-subalgebra of $\cal A_{[0,1]}$ generated by the image of $I$, we obtain a non-degenerate inclusion $\iota: I \hookrightarrow J$. Now, $\cal A_{[0,1]}$ is nuclear and $\cal O_2$-stable by \cite[Proposition 6.1]{KirchbergRordam2005purely} and its ideal lattice is order isomorphic to $[0,1]$ by \cite[Proposition 2.1]{Rordam_ASHalgebra}, so it follows from \cite[Theorem B]{BGSW_2022nuclear} that $J$ is stable. Moreover, $\cal O_2$-stability of $\cal A_{[0,1]}$ passes to $J$.

Since the inclusion $\iota$ is non-degenerate, it induces a pushforward extension $\mathfrak e_0$
\begin{equation}
\begin{tikzcd}
    \mathfrak e: 0 \arrow[r] & I \arrow[r] \arrow[d,"\iota"] & E \arrow[r] \arrow[d] & D \arrow[r] \arrow[d,equals] & 0 \\
    \mathfrak e_0: 0 \arrow[r] & J \arrow[r] & E_0 \arrow[r] & D \arrow[r] & 0
\end{tikzcd}
\end{equation}
where the map $E\to E_0$ is injective due to $\iota$ being injective. By the previous lemma, we know that the extension
\begin{equation}
\begin{tikzcd}
    \mathfrak f: 0 \arrow[r] & J \otimes D \arrow[r] & J^\dagger \otimes D \arrow[r] & D \arrow[r] & 0
\end{tikzcd}
\end{equation}
is purely large. Moreover, it follows from Corollary \ref{equivalences} and Proposition \ref{prop: embedding tensor product} that ${J^\dagger \otimes D}$ embeds into $(\Z\otimes\Z)_\omega \cong \Z_\omega$, as $\Z$ is strongly self-absorbing \cite[Theorem~7.6 and 8.8]{JiangSu1999}. Since $J$ is stable and $\cal O_2$-stable and $D$ is simple and nuclear, the $\cal O_2$-absorption theorem \cite{kirchberg2000embedding} implies that $J\otimes D \cong J$, so both $\mathfrak e_0$ and $\mathfrak f$ are extensions of $D$ by $J$. The goal is now to show that $\mathfrak e_0$ and $\mathfrak f$ are strongly unitarily equivalent. This will imply that the extension algebras $E_0$ and $J^\dagger \otimes D$ are isomorphic, giving the desired embedding of $E$ into $\Z_\omega$.

First assume $D$ is non-unital. As $\mathfrak f$ is purely large, it is absorbing by \cite[Theorem 2.1]{Gabe_nonunitalext}. Moreover, we can assume $\mathfrak e_0$ is absorbing without loss of generality. Indeed, taking a split, absorbing extension $\mathfrak f_0$ of $D$ by $J$, the sum $\mathfrak e_0 \oplus \mathfrak f_0$ is absorbing, and its extension algebra is given by
\begin{equation}
    M = \{ (m, d) \in \cal M(J) \oplus D \mid \pi(m) = \tau_e(d)\oplus\tau_f(d)\},
\end{equation}
where $\pi: \cal M(J) \to \cal Q(J)$ is the quotient map, $\tau_e, \tau_f: D \to \cal Q(J)$ are the Busby maps corresponding to $\mathfrak e_0$ resp.\ $\mathfrak f_0$ and $\tau_e(d)\oplus\tau_f(d)$ denotes the direct sum of $\tau_e(d)$ and $\tau_f(d)$ in $\cal Q(J)$ induced by the direct sum in $\cal M(J)$. We can then embed $E_0$ into $M$ by noting that the $*$-homomorphism
\begin{equation}
    E_0 \to \cal M(J) \oplus D: x \mapsto (\sigma_e(x)\oplus(\sigma_f \circ \phi \circ q)(x), q(x))
\end{equation}
is injective and maps into $M$, where $q: E_0\to D$ is the quotient map, $\phi:D\to E_0$ is a splitting of $\mathfrak f_0$ and $\sigma_e, \sigma_f: E_0\to \cal M(J)$ are the canonical maps descending to $\tau_e$ resp.\ $\tau_f$. This means that it suffices to embed $M$ into $\Z_\omega$, so we can replace $\mathfrak e_0$ by the absorbing extension $\mathfrak e_0 \oplus \mathfrak f_0$.

In summary, both $\mathfrak e_0$ and $\mathfrak f$ are absorbing extensions in $\mathrm{Ext}(D,J)$. As $\mathrm{Ext}(D,J) \cong KK^1(D,J) = 0$ by \cite[\S 7, Theorem 1]{Kasparov} and the fact that $\cal O_2$ is $KK$-contractible, $\mathfrak e_0$ and $\mathfrak f$ must be strongly unitarily equivalent, as desired.

Now suppose $E$ is unital. It then follows from \cite{Elliott-Kucerovsky_2001} that $\mathfrak f$ is a unitally absorbing extension. By using the same trick as in the non-unital case (instead adding a unital, unitally absorbing extension to $\mathfrak e_0$), we can also arrange that $\mathfrak e_0$ is a unital, unitally absorbing extension. Furthermore, due to $\mathrm{Ext}(D,J)$ and $K_0(J)$ being trivial, it follows from \cite[Theorem 3]{ManuilovThomsen} that $\mathrm{Ext}_{us}(D,J)$ is also trivial. Hence, $\mathfrak e_0$ and $\mathfrak f$ must be strongly unitarily equivalent.

For the final part of the statement, note that the embedding $E \to E_0$ is unital whenever $E$ is unital. If the embedding of $D$ into $\Z_\omega$ is unital, then the same holds for the embedding $J^\dagger\otimes D\hookrightarrow \Z_\omega$, as mentioned in Proposition \ref{prop: embedding tensor product}. Therefore, the embedding $E\hookrightarrow \Z_\omega$ constructed above is indeed unital.
\end{proof}

\begin{remark}
    The unitality/non-unitality assumption in the statement is necessary, as non-unital extensions of unital, $\Z_\omega$-embeddable $C^*$-algebras are not $\Z_\omega$-embeddable in general. For example, consider the extension
    \begin{equation}
    \begin{tikzcd}
        \mathfrak e: 0 \arrow[r] & CM_2 \arrow[r] & E \arrow[r] & \mathbb C \arrow[r] & 0,
    \end{tikzcd}
    \end{equation}
    where
    \begin{equation}
        E := \{ f \in C([0,1],M_2) \mid f(0) \in \mathbb C \oplus0 \},
    \end{equation}
    and the quotient map is evaluation at zero. This extension satisfies all the required conditions aside from (non-)unitality. However, $E$ contains several different non-zero projections, such as the elements defined by $f(t) = 1\oplus 0$ ($t\in [0,1]$) and
    \begin{equation}
        g(t) = \begin{pmatrix} 1-t & \sqrt{t(1-t)} \\ \sqrt{t(1-t)} & t
        \end{pmatrix} \quad (t\in [0,1]).
    \end{equation}
    This means that $E$ cannot embed into $\Z_\omega$, as the latter is projectionless.
\end{remark}

\begin{cor}
    Let $E$ be a continuous field of $C^*$-algebras over a connected, compact Hausdorff space $X$. If $E$ is separable and exact, and one of the fibers is simple, nuclear and $\Z_\omega$-embeddable, then $E$ is $\Z_\omega$-embeddable.
\end{cor}
\begin{proof}
Suppose the fiber $E_x$ at $x\in X$ is simple, nuclear and $\Z_\omega$-embeddable. As $X$ is compact and connected, the short exact sequence
\begin{equation}
\begin{tikzcd}
    \mathfrak e: 0 \arrow[r] & I \arrow[r] & E \arrow[r] & E_x \arrow[r] & 0
\end{tikzcd}
\end{equation}
satisfies the required conditions for the previous theorem.
\end{proof}

\begin{defn}
    Let $A$ and $B$ be $C^*$-algebras. We say $A$ is \emph{homotopically dominated by $B$} if there exist $*$-homomorphisms $\varphi:A\to B$ and $\psi: B\to A$ such that $\psi\circ \varphi \sim_h \id_A$. We say $A$ is \emph{homotopy equivalent to $B$} if moreover $\varphi\circ \psi \sim_h \id_B$.
\end{defn}

\begin{cor}
    Let $A$ and $B$ be separable $C^*$-algebras with $A$ exact and $B$ simple and nuclear. If $A$ is homotopically dominated by $B$ and $B$ embeds into $\Z_\omega$, then so does $A$. In particular, if a separable, exact $C^*$-algebra $A$ is homotopy equivalent to $\Z$, then it embeds into $\Z_\omega$.
\end{cor}
\begin{proof}
Let $\varphi:A\to B$ and $\psi: B\to A$ be $*$-homomorphisms such that $\psi\circ \varphi \sim_h \id_A$; say $\Phi: A\to C([0,1],A)$ is a $*$-homomorphism such that $\Phi_0 = \id_A$ and $\Phi_1 = \psi\circ \varphi$. The mapping cone of $\psi$,
\begin{equation}
    E := \{ (f,b) \in C([0,1],A)\oplus B \mid f(1) = \psi(b)\},
\end{equation}
is a separable and exact continuous field of $C^*$-algebras over $[0,1]$. Moreover, the fiber at $1$ is isomorphic to $B$, so the previous corollary applies and $E$ embeds into $\Z_\omega$.

Now, we can embed $A$ into $E$ through the $*$-homomorphism $\Phi\oplus\varphi$. Note that it maps into $E$ as $\Phi(a)(1) = \psi\circ\varphi(a)$ ($a\in A$) and is injective because $\Phi(a)(0) = a$ ($a\in A$). Hence, we can conclude that $A$ also embeds into $\Z_\omega$.
\end{proof}

\printbibliography

\section*{Open Access and Data Statement}
For the purpose of Open Access, the authors have applied a CC BY public copyright licence to any Author Accepted Manuscript (AAM) version arising from this submission.

Data Access Statement: Data sharing is not applicable to this article as no new data were created or analysed in this work.
\end{document}